\documentclass[11pt]{article}
\usepackage{amsmath}
\usepackage{amstext,amsbsy}
\usepackage{epsfig,multicol}
\usepackage{amsthm}
\usepackage{amssymb,latexsym}
\usepackage{color}
\usepackage{tikz}
\usepackage{enumerate}
\usepackage{framed}
\usepackage{caption}
\usepackage{subcaption}
\usetikzlibrary{decorations.pathreplacing}

\usepackage{xcolor}
\usepackage{mathdots}
\usepackage{yhmath}
\usepackage{booktabs}
\usepackage{blkarray}
\usepackage[all]{xy}
\usepackage{amsfonts}
\usepackage[english]{babel}
\usepackage{mathtools}
\usepackage[]{algorithm2e}
\usepackage{MnSymbol}

\newtheorem{thm}{Theorem}

\newtheorem{exa}[thm]{Example}
\newtheorem{lem}[thm]{Lemma}
\newtheorem{cor}[thm]{Corollary}

\newtheorem{conj}[thm]{Conjecture}

\newtheorem{rem}[thm]{Remark}

\begin{document}
\title{Properties of the Promotion Markov Chain on Linear Extensions}
\author{
Svetlana Poznanovi\'c and Kara Stasikelis \\ [6pt]
Department of Mathematical Sciences\\
Clemson University, Clemson, SC 29634, USA\\[5pt]
}
\date{} 
\maketitle
\begin{abstract} The Tsetlin library is a very well studied model for the way an arrangement of books on a library shelf evolves over time.  One of the most interesting properties of this Markov chain is that its spectrum can be computed exactly and that the eigenvalues are linear in the transition probabilities. This result has been generalized in different ways by various people. In this work we investigate one of the generalizations given by  the extended promotion Markov Chain on linear extensions of a poset $P$ introduced by Ayyer, Klee, and Schilling in 2014. They showed that if the poset $P$ is a rooted forest, the transition matrix of this Markov chain has eigenvalues that are linear in the transition probabilities and described their multiplicities. We show that the same property holds for a larger class of posets for which we also derive convergence to stationarity results.
\end{abstract} 

 \noindent{\bf Keywords: extended promotion, linear extension, Markov chain, eigenvalues} 


{\renewcommand{\thefootnote}{} \footnote{\emph{E-mail addresses}:
spoznan@clemson.edu (S.~Poznanovi\'c), stasike@g.clemson.edu (K.~Stasikelis)}

\footnotetext[1]{The first author is partially supported by NSF grant DMS-1312817. } 

\section{Introduction} \label{S:introduction}

In this paper we study the promotion Markov chain on the set $\mathcal{L}(P)$ of linear extensions of a poset $P$. The moves used are a generalization of the Sch\"{u}tzenberger's promotion operator on $\mathcal{L}(P)$, hence the name. This Markov chain was introduced by Ayyer, Klee, and Schilling~\cite{cmc}, where they showed that if the Hasse diagram of $P$ is a rooted forest, then the transition matrix  has eigenvalues which are linear in the transition probabilities. They noticed, however, that their result does not classify all the posets with this  nice property.  The main goal of  this paper is to provide a larger class of posets for which the same result holds. 

The promotion Markov chain can be viewed as a generalization of the Tsetlin library~\cite{tsetlin1963finite}, a model for the way an arrangement of books on a library shelf evolves over time. In this Markov chain on permutations of $n$ books, a book $i$  is picked up and put at the back of the shelf with probability $x_{i}$. Due to its use in computer science, this is a very well-studied Markov chain. The spectrum of its transition matrix is of particular interest because, in general, the eigenvalues can give some indication about the rate of convergence to stationarity. Hendricks~\cite{hendricks1972stationary, hendricks1973extension} found the stationary distribution, while the fact that the eigenvalues have an elegant formula (they are all sums of the transition probabilities $x_{i}$) was independently discovered by Donnelly~\cite{donnelly1991heaps}, Kapoor and Reingold~\cite{kapoor1991stochastic}, and Phatarfod~\cite{phatarfod1991matrix}. Fill derived the transition probabilities for any number of steps in~\cite{fill1996exact} where also convergence to stationarity is discussed. 

The results about the spectrum of the Tsetlin library have been generalized in various ways. For example,  Bidigare et al.~\cite{bidigare1999combinatorial} showed that this is a special case of a walk on hyperplane arrangements (also studied in~\cite{brown1998random}), which in general has nice eigenvalues. This was further generalized to a class of monoids called left regular bands~\cite{brown2000semigroups} and subsequently to all bands~\cite{brown2004semigroup} by Brown, due to the fact that these kinds of results can be seen as a consequence of the representation theory for the monoid generated by the moves in the Markov chain. This theory has later been used by Bj\"orner~\cite{bjorner2008random, bjorner2009note} to extend eigenvalue formulas for the Tsetlin library from a single shelf to hierarchies of libraries and complex hyperplane arrangements. More recently, Ayyer et al.~\cite{ayyer2015markov} extended the results to the wider class of $\mathcal{R}$-trivial monoids and obtained the description of the eigenvalues  of the promotion Markov chain for rooted forests as a consequence of the associated  monoid being $\mathcal{R}$-trivial. 

In this paper we study the promotion Markov chain for a class of posets whose components are an ordinal sum of a rooted forest and what we call a ladder. The associated monoid is not $\mathcal{R}$-trivial, so one can not use the same arguments as in the case of rooted forests to find its spectrum. However, we show that for these posets, the eigenvalues of the transition matrix  are also linear in the probabilities $x_{i}$ of the moves (Theorem~\ref{forestladdereig}). We also give a way to compute the eigenvalues explicitly (Theorem~\ref{generalfactor}). 

The outline of the paper is as follows. We start by giving the needed definitions and background in Section~\ref{S:notation}. In Section~\ref{S:oneladder} we first show that when $P$ is a single ladder, the transition matrix is diagonalizable and we find its eigenfunctions. While the transition matrix of the Tsetlin library is is diagonalizable, this is not true for general forests. Then we prove Theorem~\ref{forestladdereig} in Section~\ref{S:proofmain}.  In Section~\ref{S:convergencesec} we derive the partition function for our class of posets and convergence results for the case when $P$ has a single component. Finally, we finish with a discussion about other posets in Section~\ref{S:conclusion}.


\section{Background and main results} \label{S:notation}

Consider a poset $P$ on the set $[n] = \{1,2,\ldots,n\}$, with partial order $\preceq$.  A \emph{linear extension} of $P$ is a total ordering $\pi = \pi_1 \cdots \pi_n$ of its elements such that $\pi_i \prec \pi_j$ implies $i < j$.  The set of linear extensions of $P$ is denoted by $\mathcal{L}(P)$. 

 Ayyer, Klee, and Schilling \cite{cmc,mcpo} introduced the idea of an extended promotion operator $\partial_i$ on $\mathcal{L}(P)$. This generalizes Sch{\"u}tzenberger's \cite{schutzenberger1972promotion} promotion operator, $\partial$, which can be expressed in terms of more elementary operators $\tau_{i}$ as shown in~\cite{haiman1992dual, malvenuto1994evacuation}.  Namely, for $i=1, \ldots,n-1$ and $\pi=\pi_1 \cdots \pi_n \in \mathcal{L}(P)$, let 
 \[\tau_i \pi = \begin{cases} \pi_1 \cdots \pi_{i-1}\pi_{i+1} \pi_i \cdots \pi_n & \text{if } \pi_i \text{ and } \pi_{i+1} \text{ are incomparable in } P, \\ \pi & \text{otherwise}. \end{cases}\]
 In other words, $\tau_i$ acts nontrivially if the interchange of $\pi_i$ and $\pi_{i+1}$ yields a linear extension of $P$. The \emph{extended promotion operator} $\partial_i$, $1 \leq i \leq n$, on $\mathcal{L}(P)$ is defined by \[\partial_i = \tau_{n-1} \cdots \tau_{i+1}\tau_i.\] In particular, $\partial_{1} = \partial$ and $\partial_{n}=\mathrm{id}_{\mathcal{L}(P)}$. Note that in this paper the operators act from the left; so  $\tau_i$ is applied first, then $\tau_{i+1}$, etc.
 
 The \emph{promotion graph} is an edge-weighted directed graph $G_{P}$ whose vertices are labeled by the elements of $\mathcal{L}(P)$. $G_{P}$ contains a directed edge from $\pi$ to $\pi'$, with edge weight $x_{\pi_i}$, if and only if $\pi' = \partial_i \pi$. If  $x_i \geq 0$,  $i =1, \ldots, n$ and  $\sum_{i=1}^n x_i=1$, this gives  rise to the promotion Markov chain on $\mathcal{L}(P)$, whose row stochastic transition matrix we will denote by $M^P$.

 \begin{exa}\label{promotionex}
Consider the poset $P$ from Figure~\ref{posetex}. 

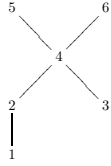
\begin{figure}[h]
\begin{center}
\scalebox{.5}{$$\xymatrix{ 5 \ar@{-}[dr] & & 6 \ar@{-}[dl] \\ & 4 \ar@{-}[dr] \ar@{-}[dl] & \\ 2 \ar@{-}[d] & & 3 \\ 1 & & }$$}
\end{center}
\caption{An example of an ordinal sum of a forest and a ladder.}
\label{posetex}
\end{figure}

The linear extensions of $P$ are  \[\mathcal{L}(P) = \{123456, 123465, 132456,132465, 312456,312465\}.\] For $\pi = 312465 \in \mathcal{L}(P)$, \[\partial_3\pi = \tau_5\tau_4\tau_3312465 = \tau_5\tau_4 312465 = \tau_5 312465  = 312456.\] Thus, since $\pi_3 = 2$, in $G_{P}$  there is a directed edge from 312465 to 312456 with edge weight $x_2$. The promotion graph $G_{P}$ is given in Figure~\ref{Pgforexample}.

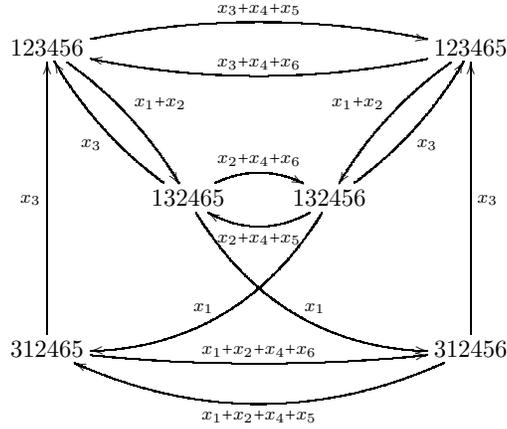
\begin{figure}[h]
\begin{center}
\scalebox{.85}{$$\xymatrix{ 123456\ar@{->}@/^1pc/[rrr]^{x_3+x_4+x_5}\ar@{->}@/^/[ddr]^{x_1+x_2} & & & 123465\ar@{->}@/^1pc/[lll]_{x_3+x_4+x_6}\ar@{->}@/_/[ddl]_{x_1+x_2} \\  & & & \\ & 132465\ar@{->}@/^/[uul]^{x_3}\ar@{->}@/^1pc/[r]^{x_2+x_4+x_6}\ar@{->}@/_1.9pc/[ddrr]^{x_1}& 132456\ar@{->}@/_/[uur]_{x_3}\ar@{->}@/^1pc/[l]^{x_2+x_4+x_5}\ar@{->}@/^1.9pc/[ddll]_{x_1} & \\  &  & & \\ 312465\ar@{->}@/_.5pc/[rrr]^{x_1+x_2+x_4+x_6}\ar@{->}[uuuu]^{x_3} & & & 312456\ar@{->}@/^2pc/[lll]^{x_1+x_2+x_4+x_5}\ar@{->}[uuuu]_{x_3} }$$}
\end{center}
\caption{Promotion graph of the poset from Figure~\ref{posetex}. Self-loops are omitted. Instead of multiple edges between vertices we have drawn only one edge with edge weights added.}
\label{Pgforexample}
\end{figure}

With the lexicographic ordering of the elements of $\mathcal{L}(P)$, the transition matrix of the promotion Markov chain is
{\tiny\[M^P = \begin{pmatrix} x_6 & x_3 + x_4 + x_5 & 0 & x_1 + x_2 & 0 & 0 \\ x_3 + x_4 + x_6 & x_5 & x_1 + x_2 & 0 & 0 & 0 \\ 0 & x_3 & x_6 & x_2 + x_4 + x_5 & 0 & x_1 \\ x_3 & 0 & x_2 + x_4 + x_6 & x_5 & x_1 & 0 \\ 0 & x_3 & 0 & 0 & x_6 & x_1 + x_2 + x_4 + x_5 \\ x_3 & 0 & 0 & 0 & x_1 + x_2 + x_4 + x_6 & x_5 \end{pmatrix}.\]}
\end{exa}

 A \emph{rooted tree} is a connected poset in which each vertex has at most one successor. A union of rooted trees is called a \emph{rooted forest}. An \emph{upset} (or \emph{upper set}) $S$ in a poset is a subset such that if $x \in S$ and $y \succeq x$, then $y \in S$.  Consider a poset $P$ with minimal element $\hat{0}$ and maximal element $\hat{1}$, then for each element $x \in P$, the \emph{derangement number} of $x$~\cite{brown2000semigroups} is
\[d_x = \displaystyle \sum_{y \succeq x } \mu(x,y)f([y, \hat{1}]),\]
where $f([y, \hat{1}])$ is the number of maximal chains in the interval $[y, \hat{1}]$ and $\mu$ is the M{\"o}bius function~\cite{sta97}. One poset with the corresponding derangement numbers is given later in Example~\ref{relaxex}.  

 
One of the main results in \cite{cmc} is that for a rooted forest $P$, the characteristic polynomial of $M^P$  factors into linear terms. 
\begin{thm}\label{forest}\cite{cmc} Let $P$ be a rooted forest of size $n$ and let $M^{P}$ be the transition matrix of the promotion Markov chain. Then \[\det(M - \lambda I_n) = \prod_{ \overset{S\subseteq [n]}{S \text{ upset in }P}} (\lambda - x_{S})^{d_S},\] where $x_S = \sum_{i \in S} x_i$ and $d_S$ is the derangement number in the lattice $L$ (by inclusion) of upsets in $P$. 
\end{thm}
A linear extension $\pi$ of a naturally labeled poset is called a \emph{poset derangement} if it has no fixed points when considered as a permutation. Let $\mathfrak{d}_P$ be the number of poset derangements of the naturally labeled poset $P$.  If $P$ is a union of chains, the eigenvalues of $M^{P}$ have an alternate description. 
\begin{thm}\label{chain}\cite{cmc} Let $P=[n_1]+[n_2] + \cdots +[n_k]$ be a union of chains of size $n$ whose elements are labeled consecutively within chains. Then $$\det(M - \lambda I_n) = \displaystyle \prod_{\underset{S \text{ upset in } P}{S \subseteq [n]}} (\lambda - x_S)^{\mathfrak{d}_{P\setminus S}}$$
where $\mathfrak{d}_\emptyset =1$. 
\end{thm}
The work of Ayyer et al. doesn't fully classify the posets with nice properties. For example, the poset from Figure~\ref{ladder} has eigenvalues $x_{1}+x_{2}+x_{3}+x_{4}, 0, x_{3}+x_{4}, -(x_{1}+x_{2})$. Notice that, unlike in the case of forests, some of the eigenvalues contain negative coefficients. In view of this, they made the following conjecture.
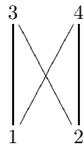
\begin{figure}[h]
\begin{center}
\scalebox{0.7}{$$\xymatrix{ 3\ar@{-}[dd]\ar@{-}[ddr] & 4\ar@{-}[dd]\ar@{-}[ddl] \\ & \\  1 & 2 }$$}
\end{center}
\caption{A  ladder of rank 2.}
\label{ladder}
\end{figure}

If $x \prec y$ we say that $y$ is a \emph{successor} of $x$.
\begin{conj}\label{conj}\cite{mcpo}
Let $P$ be a poset of size $n$ which is not a down forest and $M^{P}$ be its promotion transition matrix. If $M^{P}$ has eigenvalues which are linear in the parameters $x_{1}, \ldots, x_{n}$, then the following hold
\begin{enumerate}[(1)]
\item \label{cone} the coefficients of the parameters in the eigenvalues are only one of $\pm 1$,
\item \label{ctwo} each element of $P$ has at most two successors,
\item \label{cthree} the only parameters whose coefficients in the eigenvalues are $-1$ are those which either have two successors or one of whose successors has two successors.
\end{enumerate}
\end{conj}
Even though we have not managed to fully classify the posets with nice properties, our results give further support to $\eqref{cone}$ and $\eqref{ctwo}$ from Conjecture~\ref{conj}, but show that~\eqref{cthree} is not true (Example~\ref{relaxex}).  

Let $P$ and $Q$ be two posets. The \emph{direct sum} of $P$ and $Q$  is the poset $P + Q$ on their disjoint union such that $x \preceq y$ in $P+Q$ if either (a) $x, y \in P$ and $x \preceq y$ in $P$ or (b) $x, y \in Q$ and $x \preceq y$ in $Q$. The \emph{ordinal sum} of $P$ and $Q$ is the poset $P \oplus Q$  on their disjoint union such that $x \preceq y$ in $P \oplus Q$ if (a) $x, y \in P$ and $x \preceq y$ in $P$, or (b) $x, y \in Q$ and $x \preceq y$ in $Q$,  or (c) $x \in P$ and $y \in Q$.
We will say that the poset $P$ is a \emph{ladder} of rank $k$ if $P = Q_1 \oplus \cdots \oplus Q_k$ where $Q_{i}$ is an antichain of size 1 or 2 for all $i= 1, \ldots, k$. For example, the poset from Figure~\ref{ladder} is a ladder of rank 2, while the poset from Figure~\ref{posetex} is an ordinal sum of a forest on $\{1,2,3\}$ and a ladder on $\{4,5,6\}$.

Our main result is the following.
\begin{thm}\label{forestladdereig}
Let $F_i$ be a rooted forest and let $L_i$ be a ladder for $i =1, \ldots, k$. The eigenvalues of the promotion transition matrix $M^P$ for $P = F_1 \oplus L_1 + \cdots  + F_k \oplus L_k $ are linear in $x_1, \ldots, x_n$. Moreover, they can be explicitly computed using the formula for the eigenvalues of forests (Theorem~\ref{forest}) and Theorem~\ref{generalfactor}.
\end{thm}

The idea behind our proof is that the poset $P = F_1 \oplus L_1 + \cdots  + F_k \oplus L_k$ with $|L_{i}|=n_{i}$ can be obtained  by starting with a poset $P'=F_1 \oplus C_{n_{1}} + \cdots  + F_k \oplus C_{n_{k}}$, where $C_{i}$ is a chain of size $i$, and breaking covering relations in the chains $C_{i}$ one by one. In Theorem~\ref{generalfactor}, we show how the eigenvalues of the intermediary posets are related. Notice that $P'$ is a forest. Therefore, using Theorem~\ref{generalfactor}, the eigenvalues of $M^{P}$ and their multiplicities can be obtained from the eigenvalues of $M^{P'}$ given by Theorem~\ref{forest}. If $P$ is just a union of ladders, as a starting point one could use the simpler description of the eigenvalues and their multiplicities for a union of chains given in  Theorem~\ref{chain}.

A Markov chain is said to be \emph{irreducible} if the associated directed graph is strongly connected. In addition, it is said to be \emph{aperiodic} if the greatest common divisor of the lengths of all possible loops from any state to itself is one. For irreducible aperiodic chains, the Perron-Frobenius theorem guarantees that there is a unique stationary distribution. Ayyer et al.~\cite{cmc} showed that the promotion Markov chain is irreducible and aperiodic and obtained the following result about its stationary distribution.

\begin{thm}\cite{cmc}\label{stationary}
The stationary state weight of the linear extension $\pi \in  \mathcal{L}(P)$ for the discrete-time Markov chain for the promotion graph
is proportional to 
 \[ w(\pi) = \displaystyle \prod_{i=1}^n \frac{1}{x_{\pi_1}+ \cdots + x_{\pi_i}}.\]
\end{thm}
These weights do not necessarily sum up to 1, which is remedied by multiplication by a suitable factor $Z_{P}$, known as the \emph{partition function}. In~\cite{cmc}, the authors found $Z_{P}$  and in~\cite{mcpo} they derived results about convergence to stationarity for rooted forests. In Section~\ref{S:convergencesec}, we describe the partition function when $P=F_1 \oplus L_1 + \cdots  + F_k \oplus L_k$ is a union of ordinal sums of forests and ladders, and derive convergence results for the case when $P=F \oplus L$.

\section{The case of one ladder} \label{S:oneladder}



In this section we show that when $P$ is a ladder, the promotion transition matrix $M^{P}$ is diagonalizable and we explicitly describe its eigenvalues and eigenfunctions. We note that in general, $M^{P}$ is not diagonalizable if $P$ is a forest or a union of two or more ladders. Let $I_n$ denote the identity  matrix of size $n$ and $J_{n}$ be the anti-diagonal matrix of size $n$
\[J_n = \begin{pmatrix}  0 & 0 & 1 \\ 0 & \udots  & 0 \\ 1 & 0  & 0 \end{pmatrix}.\]

\begin{lem}\label{matrixformlem} Let $P$ be a poset of size $n$ and let $Q$ be an antichain of size $j \in \{1,2\}$. Then $$M^{P \oplus Q} = M^P \otimes J_j +I_N \otimes M^Q,$$ where $N = | \mathcal{L}(P)|$.  
\end{lem}
\begin{proof}
Let first $Q= \xymatrix{\bullet a}$. Then $M^Q =\begin{pmatrix}x_a\end{pmatrix}$ and $\mathcal{L}(P \oplus Q) = \{\pi a \colon \pi \in \mathcal{L}(P)\}.$ One can readily see that $\pi a \xrightarrow{x_a} \pi a$ and $\pi a \xrightarrow{x_j} \pi'a$ in the promotion graph $G_{P \oplus Q}$ if and only if $\pi \xrightarrow{x_j} \pi'$ in $G_P$, $j=1,\ldots, n$. Therefore,
$$M^{P \oplus Q} = M^P + x_{a}I_{N} = M^P \otimes J_1 + I_N \otimes M^Q.$$
Let now $Q =\xymatrix{\bullet a & \bullet b}$. Then $M^Q = \begin{pmatrix} x_b & x_a \\ x_b & x_a \end{pmatrix}$ and $\mathcal{L}(P \oplus Q) = \{\pi ab, \pi ba \colon \pi \in \mathcal{L}(P)\}.$ The matrix $M^{P \oplus Q}$ is of size $2N$ with blocks
\begin{center}
\begin{tabular}{r|c c|   }
 \multicolumn{1}{r}{}& \multicolumn{1}{r}{$\pi ab $}  &  \multicolumn{1}{r}{$\pi ba $}\\
\cmidrule{2-3}
$\pi ab$ & $x_b$ & $x_a$ \\
$\pi ba$ & $x_b$  & $x_a$  \\
\cmidrule{2-3}
\end{tabular}
\end{center}
on the diagonal. Furthermore, for $j \neq a, b$, if $\pi \xrightarrow{x_j} \pi'$ in $G_P$, then in  $M^{P \oplus Q}$ we have
\begin{center}
\begin{tabular}{r|c c|   }
 \multicolumn{1}{r}{}& \multicolumn{1}{r}{$\pi' ab $}  &  \multicolumn{1}{r}{$\pi' ba $}\\
\cmidrule{2-3}
$\pi ab$ &0  & $x_j$ \\
$\pi ba$ & $x_j$  & 0 \\
\cmidrule{2-3}
\end{tabular}.
\end{center}
Thus, $M^{P \oplus Q} = M^P \otimes J_2 + I_N \otimes M^Q$.
\end{proof}

\begin{cor} \label{cordecomp} Let $P = Q_1 \oplus \cdots \oplus Q_k$ be a rank $k$ ladder and let 
\[B_i  = \begin{cases} \begin{pmatrix} x_{b_i} & x_{a_i} \\ x_{b_i} & x_{a_i}  \end{pmatrix}  & \text{ if } Q_i = \xymatrix{\bullet a_{i} & \bullet b_{i}} \\  (x_{a_i}) & \text{ if } Q_i = \xymatrix{\bullet a_{i}}. \end{cases}\] Then \[M^P= \displaystyle \sum_{t=1}^k I_{|Q_{1}|} \otimes \cdots  \otimes I_{|Q_{t-1}|} \otimes B_{t} \otimes J_{|Q_{t+1}|} \otimes \cdots  \otimes J_{|Q_{k}|}.\]
\end{cor}

\begin{proof} Since $M^{Q_{i}} = B_{i}$, the claim follows by iteratively applying Lemma~\ref{matrixformlem}.
\end{proof}

To describe the  eigenvalues and eigenfunctions of $M^{P}$ for a ladder $P = Q_1 \oplus \cdots \oplus Q_k$, we consider the set of vectors $v$ and corresponding scalars $c^{v}$ that can be obtained as follows.

\begin{algorithm}[H]
 $c_0 = 0$\\
 \For{$i = 1$ \KwTo $k$}{
  \If{$|Q_i|=1$}{
    $v_i = (1)$\\
    $c_{i} = c_{i-1} + x_{a_{i}}$
  } 
  \If{$|Q_i|=2$}{
    $v_i = \begin{pmatrix}1 \\1 \end{pmatrix}$\\
$c_{i} = c_{i-1} +x_{a_{i}} + x_{b_{i}}$\\

or \\

$v_{i} = \begin{pmatrix*}[r]-x_{a_i} \\ x_{b_i} \end{pmatrix*} - c_{i-1} \begin{pmatrix*}[r] 1 \\ -1 \end{pmatrix*} $\\
$c_{i} = - c_{i-1}$
   }
}
$v = v_{1} \otimes \cdots \otimes v_{k}$\\
$c^v = c_k$
\caption{Algorithm for finding the eigenvalues and eigenfunctions of a ladder.}
\label{algo}
\end{algorithm}

For example, the vectors $v$ that can be generated this way for the ladder $P$ from Figure~\ref{ladder} are

\[ \begin{pmatrix}1 \\ 1 \end{pmatrix} \otimes \begin{pmatrix} 1 \\ 1 \end{pmatrix}, \begin{pmatrix*}[r]-x_1 \\ x_2 \end{pmatrix*} \otimes \begin{pmatrix*}[r] -x_3 \\  x_4 \end{pmatrix*}, \begin{pmatrix*}[r]-x_1 \\ x_2  \end{pmatrix*} \otimes \begin{pmatrix} 1 \\ 1 \end{pmatrix}, \begin{pmatrix}1 \\ 1 \end{pmatrix} \otimes \left(\begin{pmatrix*}[r] -x_3\\ x_4 \end{pmatrix*} -(x_1 + x_2) \begin{pmatrix*}[r] 1 \\ -1 \end{pmatrix*} \right)\]
and the corresponding scalars $c^{v}$ are
\[x_1 + x_2 + x_3 + x_4, 0, x_3+x_4, -(x_1+x_2).\]

%

\begin{thm}
If $P = Q_1 \oplus \cdots \oplus Q_k$ is a ladder, then $M^P$ is diagonalizable. In particular, the eigenvalues of $M^P$ are exactly the scalars $c^v$ that can be obtained using Algorithm~\ref{algo} with corresponding eigenfunctions $v$. 
\end{thm}

\begin{proof} Let $\tilde{v_{i}} = J_{|Q_{i}|}v_{i}$. In view of Corollary~\ref{cordecomp}, it's sufficient to prove  that for $0 \leq m \leq k-1$,
\[ \sum_{t=k-m}^k  I_{|Q_1|} \otimes \cdots \otimes I_{|Q_{k-t}|} \otimes B_{k-t+1} \otimes J_{|Q_{k-t+2}|} \otimes \cdots \otimes J_{|Q_k|}(v_1 \otimes \cdots \otimes v_k) =c_{m+1}v_1 \otimes \cdots \otimes v_{m+1} \otimes \widetilde{v}_{m+2} \otimes \cdots \otimes \widetilde{v}_{k}. \]

For $m=0$,
\begin{align*}
&\displaystyle \sum_{t=k}^k  I_{|Q_1|} \otimes \cdots \otimes I_{|Q_{k-t}|} \otimes B_{k-t+1} \otimes J_{|Q_{k-t+2}|} \otimes \cdots \otimes J_{|Q_k|}(v_1 \otimes \cdots \otimes v_k)    \\
& = B_1v_1 \otimes \widetilde{v}_2 \otimes \cdots \otimes \widetilde{v}_k \\
& = \begin{cases} (x_{a_1})v_1 \otimes \widetilde{v}_2 \otimes \cdots \otimes \widetilde{v}_k & \text{ if } v_1 = (1) \\ (x_{a_1} + x_{b_1})v_1 \otimes \widetilde{v}_2 \otimes \cdots \otimes \widetilde{v}_k & \text{ if } v_1 = \begin{pmatrix} 1 \\ 1 \end{pmatrix} \\  -c_0 v_1 \otimes \widetilde{v}_2 \otimes \cdots \otimes \widetilde{v}_k  & \text { if } v_1 = \begin{pmatrix*}[r] - x_{a_1} \\ x_{b_1} \end{pmatrix*} \end{cases}\\
& = c_1 v_1 \otimes \widetilde{v}_2 \otimes \cdots \otimes \widetilde{v}_k
\end{align*}
Using the induction hypothesis, we have
\begin{align*}
& \sum_{t=k-m}^k  I_{|Q_1|} \otimes \cdots \otimes I_{|Q_{k-t}|} \otimes B_{k-t+1} \otimes J_{|Q_{k-t+2}|} \otimes \cdots \otimes J_{|Q_k|}(v_1 \otimes \cdots \otimes v_k)    \\
& =  v_1 \otimes \cdots \otimes v_{m} \otimes B_{m+1}v_{m+1} \otimes \widetilde{v}_{m+2} \otimes \cdots \otimes \widetilde{v}_k   + c_{m} v_1 \otimes \cdots \otimes v_{m} \otimes  \widetilde{v}_{m+1} \otimes \widetilde{v}_{m+2} \otimes \cdots \otimes \widetilde{v}_k\\
&= v_1 \otimes \cdots \otimes v_m \otimes (B_{m+1} v_{m+1} +c_{m} \widetilde{v}_{m+1})\otimes \widetilde{v}_{m+2} \otimes \cdots \otimes \widetilde{v}_k\\
& = \begin{cases} v_1 \otimes \cdots \otimes (c_m + x_{a_{m+1}})v_{m+1} \otimes \cdots \otimes \widetilde{v}_k & \text{ if } v_{m+1} = (1) \\   v_1 \otimes \cdots \otimes (c_m + x_{a_{m+1}}+ x_{b_{m+1}})v_{m+1} \otimes \cdots \otimes \widetilde{v}_k & \text{ if } v_{m+1} = \begin{pmatrix} 1 \\ 1 \end{pmatrix}  \\ v_1 \otimes \cdots \otimes v_{m} \otimes (-c_m v_{m+1}) \otimes \widetilde{v}_{m+2} \otimes \cdots \otimes \widetilde{v}_k & \text { if } v_{m+1} = \begin{pmatrix*}[r] - x_{a_{m+1}} \\ x_{b_{m+1}} \end{pmatrix*} - c_{m}\begin{pmatrix*}[r] 1 \\ -1 \end{pmatrix*}\end{cases}\\
& =c_{m+1}v_1 \otimes \cdots \otimes v_{m+1} \otimes \widetilde{v}_{m+2} \otimes \cdots \otimes \widetilde{v}_{k}.\end{align*}

\end{proof}

\section{Proof of Theorem~\ref{forestladdereig}} \label{S:proofmain}

 
For a poset $P$, let $R_{P}$ be the set of all pairs $(a,b)$ for which $P$ can be written in the form  \[P = Q' \oplus a\oplus b \oplus Q''+P_{2}.\]  In this section we will assume that $R_{P} \neq \emptyset$ and for a pair $(a,b) \in R_{P}$, we will denote by $P'$ the poset $P \setminus \{(a, b)\}$, i.e., the poset whose Hasse diagram is obtained from the Hasse diagram of $P$ by deleting the edge that represents the covering relation $a \prec b$.  We will say that $M^{P}$ has the \emph{upset property} if its characteristic polynomial factors into linear terms and for each eigenvalue $x^{\mathfrak{s}} = \sum c_{k}^\mathfrak{s}x_{k}$ of $M^{P}$ and a pair $(a,b) \in R_{P}$, the following two conditions are true:
\begin{enumerate}[(a)]
\item $x_{a} \in x^{\mathfrak{s}} \implies x_{b} \in x^{\mathfrak{s}}$ and $c_{a}^\mathfrak{s} = c_{b}^\mathfrak{s}$
\item $x_{b} \in x^{\mathfrak{s}}, x_{a} \notin x^{\mathfrak{s}} \implies x_{k} \notin x^{\mathfrak{s}}$ for $k \prec_{P} a$.
\end{enumerate}

Here and throughout the paper, we will use $x_{k} \in x^{\mathfrak{s}}$ to denote that $x_{k}$ appears in $x^{\mathfrak{s}}$ with a nonzero coefficient.

Note that the matrix $M^{P}$ can be written as $M^{P} = \sum x_{i} G_{i}$, where $G_{i}$ are the matrices corresponding to the extended promotion operators $\partial_{i}$.
 
 \begin{thm}\label{generalfactor}
Let $P = Q' \oplus a\oplus b \oplus Q''+P_{2}$ and $P' = P\setminus \{(a, b)\}$. Suppose the matrices $G_{i}$ are simultaneously upper-triangularizable matrices.  If  $M^{P}$ has the upset property then so does $M^{P'}.$ In particular, for each eigenvalue $x^{\mathfrak{s}} = \sum c_{k}^\mathfrak{s}x_{k}$ of $M^P$, $M^{P'}$ has two eigenvalues given by 
\[ \begin{cases} x^{\mathfrak{s}}, \displaystyle \sum_{k \npreceq_{P} b} c_{k}^\mathfrak{s}x_{k} - \sum_{k \prec_{P} a} c_{k}^\mathfrak{s}x_{k} & \; \;  \text{ if } x_{a} , x_{b} \in x^{\mathfrak{s}} \text{ or } x_{a} , x_{b} \notin x^{\mathfrak{s}}\\
\displaystyle x^{\mathfrak{s}}, x^{\mathfrak{s}} - c_{b}^\mathfrak{s}x_{b} + c_{b}^\mathfrak{s}x_{a}  & \; \;  \text{ if } x_{a} \not\in x^{\mathfrak{s}}, x_{b} \in x^{\mathfrak{s}}. 
\end{cases}
\]
 \end{thm}
 
 \begin{rem}
 The assumption that the $G_{i}$'s are simultaneously upper-triangularizable is stronger than asking that the characteristic polynomial $M^{P}$ factors into linear terms. We don't know whether this stronger assumption is necessary but we need it in our proof.
 \end{rem}
 
Notice that each poset $F_{1} \oplus L_{1} + \cdots + F_{k} \oplus L_{k}$ for forests $F_{i}$ and ladders $L_{i}$, can be obtained starting from a forest in which the upper parts of the tree components are chains and breaking covering relations in the chains. Moreover, the transition matrix of a forest satisfies the assumptions of Theorem~\ref{generalfactor} because, as proved in~\cite{cmc}, the monoid generated by the matrices $G_{i}$ is $\mathcal{R}$-trivial and the eigenvalues of the transition matrix are supported on the upsets of the forest (Theorem~\ref{forest}).  Therefore, Theorem~\ref{forestladdereig} follows from Theorem~\ref{generalfactor}.

\begin{exa} \label{relaxex} Let $P$ be the leftmost poset in Figure~\ref{figex}. Note that $R_{P} = \{(4,5), (5,6)\}$. Let $P' = P \setminus \{(5,6)\}$ and $P'' = P \setminus \{(4,5)\}$, both illustrated in Figure~\ref{figex}. 

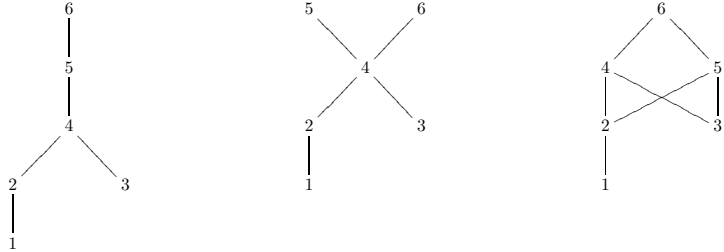
\begin{figure}[h]
\begin{center}
\scalebox{0.6}{$$\qquad \xymatrix{& 6\ar@{-}[d] & \\ & 5\ar@{-}[d] & \\ &4\ar@{-}[dl]\ar@{-}[dr] & \\ 2\ar@{-}[d] & & 3 \\ 1 & & } \qquad  \qquad \hspace{2cm} \xymatrix{ 5\ar@{-}[dr] & & 6\ar@{-}[dl] \\ & 4\ar@{-}[dl]\ar@{-}[dr] & \\2\ar@{-}[d] & & 3 \\ 1 & & } \qquad  \qquad \hspace{2cm}  \xymatrix{ & 6\ar@{-}[dl] \ar@{-}[dr]&\\  4\ar@{-}[d]\ar@{-}[drr] & & 5\ar@{-}[dll]\ar@{-}[d]\\2\ar@{-}[d] & & 3 \\ 1 & & } $$}
\end{center}
\caption{A forest $P$ and two posets  $P'$  and $P''$ obtained by breaking a covering relation in $P$.}
\label{figex}
\end{figure}
In Figure~\ref{derangement}, we give the lattice of upsets in $P$ and the corresponding derangement number for each upset. \\

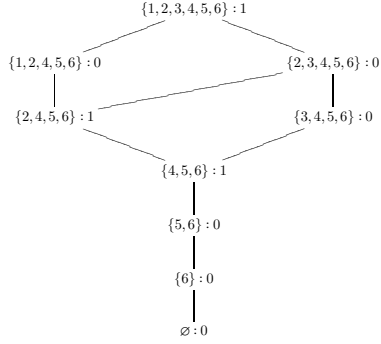
\begin{figure}[h]
\begin{center}
\scalebox{.5}{\xymatrix{
& \{1,2,3,4,5,6\}:1 \ar@{-}[dr]\ar@{-}[dl] &  \\ \{1,2,4,5,6\}:0\ar@{-}[d]&& \{2,3,4,5,6\}:0 \ar@{-}[d]\ar@{-}[dll]\\ \{2,4,5,6\}:1\ar@{-}[dr]&& \{3,4,5,6\}:0 \ar@{-}[dl]\\ &\{4,5,6\}:1 \ar@{-}[d]&\\ & \{5,6\}:0 \ar@{-}[d]& \\ &\{6\}:0\ar@{-}[d]& \\ &\emptyset: 0& }}
\end{center}
\caption{The lattice of upsets in $P$.}
\label{derangement}
\end{figure}
Thus, by Theorem~\ref{forest},  the eigenvalues of $M^{P}$ are 
\[x_4 + x_5 + x_6, x_2 + x_4  + x_5 + x_6, x_1 + x_2 + x_{3}+x_4 + x_5 + x_6\]
and, therefore, by Theorem~\ref{generalfactor}, the eigenvalues of $M^{P'}$ are
\[x_4 + x_5 + x_6, -x_4, x_2 + x_4 + x_5 + x_6, - (x_2 + x_4), x_1 + x_2 + x_3 + x_4 + x_5 +x_6, -(x_1  + x_2 + x_3 + x_4),\] 
while the eigenvalues of $M^{P''}$ are \[x_4 + x_5 + x_6,x_6 , x_2 + x_4 + x_5 + x_6, - x_2 +x_6, x_1 + x_2 + x_3 + x_4 + x_5 +x_6, -(x_1  + x_2 + x_3)+x_6.\] 
Notice that in the last eigenvalue of $M^{P'}$, $x_{1}$ appears with a negative coefficient, which contradicts property~\eqref{cthree} from Conjecture~\ref{conj}.
\end{exa}

\begin{exa}\label{relaxex2} Let $P = [4]+[1]$ be the union of a chain of size $4$ and a chain of size 1. Let $P' = P\backslash\{(2,3)\}$ (see Figure~\ref{figex2}). Note that $R_P = \{(1,2),(2,3),(3,4)\}$, $a = 2$ and $b=3$. \\

\begin{figure}[h]
\begin{center}
\scalebox{0.6}{$$\qquad \xymatrix{& 4\ar@{-}[d] & 5 \\ & 3\ar@{-}[d] & \\ &2\ar@{-}[d] & \\ & 1& } \qquad  \qquad \hspace{2cm}  \xymatrix{ & 4\ar@{-}[dl] \ar@{-}[dr]& & 5\\  2\ar@{-}[dr] & & 3\ar@{-}[dl] & \\  & 1& &  }$$}
\end{center}
\caption{A forest $P$ and the associated poset $P'$ obtained by breaking a covering relation.}
\label{figex2}
\end{figure}
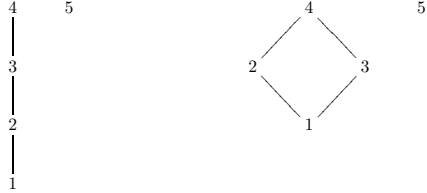
Since $P$ is a sum of chains, the multiplicities of the eigenvalues of $M^{P}$ can be computed more easily using Theorem~\ref{chain}. The eigenvalues of $M^P$ are 
\[0,x_4, x_3+x_4, x_2+x_3+x_4, x_1+x_2 + x_3 + x_4 + x_5. \]
Thus, by Theorem~\ref{generalfactor}, the eigenvalues of $M^{P'}$ are 
\[0,0,x_4,x_4,x_3+x_4, x_2+x_4, x_2+x_3+x_4, x_4, x_1+x_2+x_3+x_4+x_5, x_4+x_5 - x_1.\]

\end{exa}
The rest of this section is devoted to the proof of Theorem~\ref{generalfactor} which is based on several lemmas that we prove first. For the posets $P$ and $P'$ described at the beginning of this section, and $\pi \in \mathcal{L}(P)$, let $\hat{\pi} \in \mathcal{L}(P')$ be the linear extension of $P'$ obtained by interchanging $a$ and $b$. Then \[ \mathcal{L}(P') = \{ \pi, \hat{\pi} \colon \pi \in \mathcal{L}(P) \}.\] Recall that $G_{P}$ is the promotion graph of the poset $P$. The graphs $G_{P}$ and $G_{P'}$ are closely related as described in the following lemma.

\begin{lem}\label{form}
Let $P = Q' \oplus a\oplus b \oplus Q''+P_{2}$ and let $P' = P\setminus \{(a, b)\}$.

\begin{enumerate}[(1)]
\item \label{one} If $k \prec_{P} a$ and $\pi \overset{x_k}{\rightarrow} \tilde{\pi}$ in $G_{P}$, then $\pi\overset{x_k}{\rightarrow} \hat{\tilde{\pi}}$ and $\hat{\pi}\overset{x_k}{\rightarrow}\tilde{\pi}$ in $G_{P'}$.
\item \label{two} If $k \npreceq_{P} a, b$ and $\pi \overset{x_k}{\rightarrow} \tilde{\pi}$ in $G_{P}$, then $\pi\overset{x_k}{\rightarrow}\tilde{\pi}$ and $\hat{\pi}\overset{x_k}{\rightarrow} \hat{\tilde{\pi}}$ in $G_{P'}$.
\item \label{three} If $\pi \overset{x_a}{\rightarrow} \tilde{\pi}$ in $G_{P}$, then $\pi\overset{x_{a}}{\rightarrow}\hat{\tilde{\pi}}$ and $\hat{\pi}\overset{x_{b}}{\rightarrow}\tilde{\pi}$ in $G_{P'}$.
\item \label{four} If $\pi \overset{x_b}{\rightarrow} \tilde{\pi}$ in $G_{P}$, then $\pi\overset{x_b}{\rightarrow}\tilde{\pi}$ and $\hat{\pi}\overset{x_a}{\rightarrow}\hat{\tilde{\pi}}$ in $G_{P'}$.
\end{enumerate}
\end{lem}

\begin{proof} Notice that the structure of $P$ and $P'$ implies that for $x \neq a,b$, $x \prec a$ (respectively, $a \prec x$) if and only if $x \prec b$ (respectively, $b \prec x$). Let $m = \partial_{\pi^{-1}(k)}$. We split the analysis into four cases.

\eqref{one} If $k \prec a$, then $\pi \in \mathcal{L}(P)$ is of the form $\pi = A_{1}kA_{2}aBbC$. Because of the structure of $P'$,  we have that for every $x$ in $A_{2}$, $k \prec x$ implies $x \prec a$. Therefore, $\tilde{\pi} = \partial_{m} \pi = A_{1} (\partial_{1}kA_{2}) B a (\partial_{1} b C)$. In $\mathcal{L}(P')$, however, since $a$ and $b$ are incomparable, $\partial_{m} \pi = A_{1} (\partial_{1}kA_{2}) B b (\partial_{1} a C) = \hat{\tilde{\pi}}$. The last equality is true because  $\partial_{1} a C$ can be obtained from $\partial_{1} b C$ by replacing $b$ with $a$. Also, \[ \partial_{m} \hat{\pi} = \partial_{m} A_{1}kA_{2}bBaC = A_{1} (\partial_{1}kA_{2}) B a (\partial_{1} b C) = \tilde{\pi}.\]

\eqref{two} If $k \npreceq a, b$, there are three possible subcases.

$(2a)$. If $\pi = AaBbC_{1}kC_{2}$, then $\tilde{\pi} = \partial_{m} \pi = AaBbC_{1}(\partial_{1}kC_{2})$. But then, clearly, in $\mathcal{L}(P')$, $\partial_{m} \pi =\tilde{\pi}$ as well. Also, $\partial_{m} \hat{\pi} = \partial_{m} AbBaC_{1}kC_{2} = AbBaC_{1}(\partial_{1}kC_{2}) = \hat{\tilde{\pi}}$.

$(2b)$. If $\pi = AaB_{1}kB_{2}bC$, the analysis is similar to the previous case. Namely, $\tilde{\pi} = \partial_{m} \pi = AaB_{1} (\partial_{1} kB_{2}bC)$. So, in $\mathcal{L}(P')$, $\partial_{m} \pi = \tilde{\pi}$ as well.  Also, $\partial_{m} \hat{\pi} = \partial_{m} AbB_{1}kB_{2}aC = AbB_{1}(\partial_{1}kB_{2}aC) = \hat{\tilde{\pi}}$. 

$(2c)$. If $\pi = A_{1}kA_{2}aBbC$, then notice that $\partial_{1}kA_{2}$ ends with an element $c$ which is also incomparable with $a$. Therefore, $c$ will swap with $a$ and $a$ will precede $b$ in $\tilde{\pi}$. Hence, in $\mathcal{L}(P')$,  $\partial_{m} \pi =\tilde{\pi}$ as well. Now it's not hard to see that $\partial_{m} \hat{\pi} = \partial_{m} A_{1}kA_{2}bBaC = \hat{\tilde{\pi}}$.

\eqref{three} Let $\pi = AaBbC$. Then the elements in $B$ are incomparable to both $a$ and $b$ and therefore, for $m = \partial_{\pi^{-1}(a)}$,  $\tilde{\pi} = \partial_{m} \pi = ABa(\partial_{1}bC)$. However, in $\mathcal{L}(P')$,  $a$ and $b$ can swap, so $\partial_{m} \pi = ABb(\partial_{1}aC) = \hat{\tilde{\pi}}$. Also, \[\partial_{m} \hat{\pi} = \partial_{m} AbBaC = ABa(\partial_{1}bC) = \tilde{\pi}.\]

\eqref{four} In this case for $\pi = AaBbC \in \mathcal{L}(P)$, $\tilde{\pi} = \partial_{m} \pi = AaB(\partial_{1}bC)$. So, in $\mathcal{L}(P')$, $ \partial_{m} \pi = \tilde{\pi}$ as well and $\partial_{m} \hat{\pi} = \partial_{m}AbBaC = AbB(\partial_{1}aC) = \hat{\tilde{\pi}}$.
\end{proof}

Let $P$ be a poset of size $n$ of the form  $P= Q' \oplus a\oplus b \oplus Q''+P_{2}$. For the transition matrix $M^{P}$ of size $m$,  we will denote by $\partial_{a,b} M^{P}$ the $2m \times 2m$ matrix obtained by replacing each entry of $M^{P}$ by a $2 \times 2$ block using the linear extension of the map:

\begin{enumerate}
\item[]\begin{tabular}{r|c| c  c  r|c c|  c }
 \multicolumn{1}{r}{}& \multicolumn{1}{r}{$\tilde{\pi}$}  &  \multicolumn{1}{r}{} &  \multicolumn{1}{r}{}& \multicolumn{1}{r}{}
 &  \multicolumn{1}{c}{$\tilde{\pi}$}
 & \multicolumn{1}{c}{$\hat{\tilde{\pi}}$} &\\
\cmidrule{2-2}
\cmidrule{6-7}
$\pi$ & $x_{k}$ & & $\longmapsto$&   $\pi$ &  &$x_{k}$  & $  $  for $k \prec a$\\ 
\cmidrule{2-2}
 \multicolumn{1}{r}{} &  \multicolumn{1}{r}{} & \multicolumn{1}{r}{}&  \multicolumn{1}{r}{} &
 $\hat{\pi}$ &   $x_{k}$ && \\
\cmidrule{6-7}
\end{tabular}
\hspace{0.3cm}
\begin{tabular}{r|c| c  c  r|c c|  c }
 \multicolumn{1}{r}{}& \multicolumn{1}{r}{$\tilde{\pi}$}  &  \multicolumn{1}{r}{} &  \multicolumn{1}{r}{}& \multicolumn{1}{r}{}
 &  \multicolumn{1}{c}{$\tilde{\pi}$}
 & \multicolumn{1}{c}{$\hat{\tilde{\pi}}$} &\\
\cmidrule{2-2}
\cmidrule{6-7}
$\pi$ & $x_{a}$ & & $\longmapsto$&   $\pi$ &  &$x_{a}$  & \\ 
\cmidrule{2-2}
 \multicolumn{1}{r}{} &  \multicolumn{1}{r}{} & \multicolumn{1}{r}{}&  \multicolumn{1}{r}{} &
 $\hat{\pi}$ &   $x_{b}$ && \\
\cmidrule{6-7}
\end{tabular}

\item[]\begin{tabular}{r|c| c c   r|c c|  c }
 \multicolumn{1}{r}{}& \multicolumn{1}{r}{$\tilde{\pi}$}  &  \multicolumn{1}{r}{} &  \multicolumn{1}{r}{}& \multicolumn{1}{r}{}
 &  \multicolumn{1}{c}{$\tilde{\pi}$}
 & \multicolumn{1}{c}{$\hat{\tilde{\pi}}$} &\\
\cmidrule{2-2}
\cmidrule{6-7}
$\pi$ & $x_{k}$ & & $\longmapsto$&   $\pi$ & $x_{k}$ & & $  $  for $k \npreceq b$ \\ 
\cmidrule{2-2}
 \multicolumn{1}{r}{} &  \multicolumn{1}{r}{} & \multicolumn{1}{r}{}&  \multicolumn{1}{r}{} &
$\hat{\pi}$ &  & $x_{k}$ & \\
\cmidrule{6-7}
\end{tabular}
\hspace{0.3cm}
\begin{tabular}{r|c| c c   r|c c|  c }
 \multicolumn{1}{r}{}& \multicolumn{1}{r}{$\tilde{\pi}$}  &  \multicolumn{1}{r}{} &  \multicolumn{1}{r}{}& \multicolumn{1}{r}{}
 &  \multicolumn{1}{c}{$\tilde{\pi}$}
 & \multicolumn{1}{c}{$\hat{\tilde{\pi}}$} &\\
\cmidrule{2-2}
\cmidrule{6-7}
$\pi$ & $x_{b}$ & & $\longmapsto$&   $\pi$ & $x_{b}$ & & \\ 
\cmidrule{2-2}
 \multicolumn{1}{r}{} &  \multicolumn{1}{r}{} & \multicolumn{1}{r}{}&  \multicolumn{1}{r}{} &
$\hat{\pi}$ &  & $x_{a}$ & \\
\cmidrule{6-7}
\end{tabular}
\end{enumerate}

\noindent In particular, a zero entry goes to a $2 \times 2$ block of zeros.

\begin{cor} \label{formcor} Let $P = Q' \oplus a\oplus b \oplus Q''+P_{2}$ and let $P' = P\setminus \{(a, b)\}$. Then $M^{P'} = \partial_{a,b} M^{P}$ in an appropriate basis of $\mathcal{L}(P')$.
\end{cor}

\begin{exa}\label{relaxmat}
Let $P$ and $P'$ be as in Example~\ref{relaxex}. Then \begin{align*}\mathcal{L}(P) &= \{123456, 132456, 312456\} \\ \mathcal{L}(P') &= \{123456, 123465, 132456, 132465, 312456, 312465\},\end{align*}

\small{\[M^{P} = \begin{pmatrix} x_3 + x_4 + x_5 + x_6 & x_1  + x_2 & 0 \\ x_3 & x_2 + x_4 + x_5 + x_6 & x_1 \\ x_3 & 0 & x_1 + x_2 + x_4 + x_5 + x_6 \end{pmatrix} \]}
and 
{\tiny{\begin{align*}
M^{P'}& = \begin{pmatrix} x_6 & x_3 + x_4  + x_5 & 0 & x_1 + x_2 & 0 & 0 \\ x_3 + x_4 + x_6 & x_5 & x_1 + x_2 & 0 & 0 & 0 \\ 0 & x_3 & x_6 & x_2 + x_4 + x_5 & 0 & x_1 \\ x_3 & 0 & x_2 + x_4 + x_6 & x_5 & x_1 & 0 \\ 0 & x_3 & 0 & 0 & x_6 & x_1 + x_2 + x_4 + x_5 \\ x_3 & 0 & 0 & 0 & x_1 + x_2 + x_4 + x_6 & x_5 \end{pmatrix} .
\end{align*}}}
\end{exa}

For a complex matrix $S$, denote by $\partial S = S \otimes I_{2}$. So, if  $E$ is an elementary matrix of size $k$ corresponding to a row operation $R$ then $\partial E$ corresponds to performing a corresponding operation to 2 rows on a matrix of size $2k$. 

\begin{lem}\label{diagramlemma}
Let $S$ be a matrix with complex entries and $M$ a matrix whose entries are homogeneous degree 1 polynomials in $x_{1}, \ldots, x_{n}$.  Then \[(\partial S) (\partial_{a,b} M) =\partial_{a,b} (SM) \; \; \text{     and     } \; \; (\partial_{a,b} M) (\partial S) =\partial_{a,b} (MS) .\]
\end{lem}

\begin{proof}  Notice that the definition of $\partial_{a,b} M$ can be restated as
\[ \partial_{a,b} M = M\bigg \rvert_{\substack{x_{k}=0 \\ k \not \prec a} } \otimes \begin{pmatrix} 0 & 1 \\ 1 & 0\end{pmatrix} + M\bigg \rvert_{\substack{x_{k}=0 \\ k \preceq b} } \otimes I_{2} + \frac{1}{x_{a}} M\bigg \rvert_{\substack{x_{k}=0 \\ k \neq a} } \otimes \begin{pmatrix} 0 & x_{a} \\ x_{b} & 0 \end{pmatrix} + \frac{1}{x_{b}} M \bigg \rvert_{\substack{x_{k}=0 \\ k \neq b} } \otimes   \begin{pmatrix} x_{b} & 0 \\ 0 & x_{a}\end{pmatrix}.\] So, the claim follows since for a complex matrix $S$ independent of the $x_{i}$'s, \[SM\bigg \rvert_{\substack{x_{k}=0 \\ k \not \prec a} } = (SM) \bigg \rvert_{\substack{x_{k}=0 \\ k \not \prec a} },\] etc. 
\end{proof}

\begin{lem} \label{jordan} Let $M$ be a matrix whose entries are homogeneous degree 1 polynomials in $x_{1}, \ldots, x_{n}$ and let $S$ be a complex matrix such that $T = SMS^{-1}$ is upper triangular. Then the eigenvalues of $\partial_{a,b}M$ are the same as the eigenvalues of $\partial_{a,b}T$.
\end{lem}

\begin{proof} Note that $(\partial S)^{-1} = (S \otimes I_{2})^{-1} = S^{-1} \otimes I_{2} = \partial(S^{-1})$. So by Lemma~\ref{diagramlemma} we get
\[ \partial_{a,b} T = \partial_{a,b}(SMS^{-1}) = (\partial S) (\partial_{a,b} M) (\partial S^{-1}) = (\partial S) (\partial_{a,b} M) (\partial S)^{-1}.\]
Therefore, $\partial_{a,b} M$ and $\partial_{a,b} T$ are similar and thus have the same eigenvalues.
\end{proof}

\begin{proof}[Proof of Theorem~\ref{generalfactor}] By Corollary~\ref{formcor}, $M^{P'} = \partial_{a,b}M^{P}$. Let $S$ be the matrix that simultaneously upper-triangularizes the matrices $G_{i}$. Then $T = SM^{P}S^{-1}$ is an upper triangular matrix whose diagonal entries are the eigenvalues $x^{\mathfrak{s}}$ of $M^{P}$. Since $M^{P}$ satisfies the upset property, these eigenvalues are linear in the $x_{i}$'s. By Lemma~\ref{jordan} the eigenvalues of $M^{P'}$ are the same as the eigenvalues of $\partial_{a,b}T$ which is block upper-triangular with $2 \times 2$ blocks $\partial_{a,b}x^{\mathfrak{s}}$ on the main diagonal. Note that 
\[ \partial_{a,b}x^{\mathfrak{s}} = \begin{pmatrix*}[c]  \displaystyle c_{b}^{\mathfrak{s}}x_{b} + \sum_{k \nprec a,b}c_{k}^{\mathfrak{s}}x_{k}  & \displaystyle  c_{a}^{\mathfrak{s}}x_{a} + \sum_{k \prec a}c_{k}^{\mathfrak{s}}x_{k}  \\  \displaystyle c_{a}^{\mathfrak{s}}x_{b} + \sum_{k \prec a}c_{k}^{\mathfrak{s}}x_{k}  &  \displaystyle  c_{b}^{\mathfrak{s}}x_{a} + \sum_{k \nprec a,b}c_{k}^{\mathfrak{s}}x_{k}  \end{pmatrix*}.\]
Since by assumption, $M^{P}$ has the upset property, there are only two cases: $c_{a}^{\mathfrak{s}} = c_{b}^{\mathfrak{s}} =c$ and $c_{a}^{\mathfrak{s}} =0, c_{b}^{\mathfrak{s}}  \neq 0$. In the former case,
\[ \begin{pmatrix*}[r] 1 & 0 \\ -1 & 1   \end{pmatrix*} \partial_{a,b}x^{\mathfrak{s}} \begin{pmatrix*}[r] 1 & 0 \\ -1 & 1   \end{pmatrix*}^{-1} = \begin{pmatrix*}[c] x^{\mathfrak{s}} & 0 \\ 0 & \displaystyle\sum_{k \nprec a,b}c_{k}^{\mathfrak{s}}x_{k} - \sum_{k \prec a}c_{k}^{\mathfrak{s}}x_{k}   \end{pmatrix*} \]
In the latter case, by the upset property we also have that $\sum_{k \prec a}c_{k}^{\mathfrak{s}}x_{k} =0$ and, therefore, 
\[\partial_{a,b}x^{\mathfrak{s}} = \begin{pmatrix*}[c] \displaystyle c_{b}^{\mathfrak{s}}x_{b}  + \sum_{k \nprec a,b}c_{k}^{\mathfrak{s}}x_{k}  & 0 \\  0 &  \displaystyle c_{b}^{\mathfrak{s}}x_{a} + \sum_{k \nprec a,b}c_{k}^{\mathfrak{s}}x_{k}  \end{pmatrix*}.\]
This also shows that there is a real matrix $S'$ such that $S' (\partial_{a,b} T) (S')^{-1}$ is upper triangular. Consequently, $S'(\partial S) M^{P'} (S'(\partial S))^{-1}$ is upper triangular, which means that the matrices $G_{i}'$ such that $M^{P'} = \sum x_{k} G_{i}'$ are simultaneously upper-triangularizable.

Finally, notice that $R_{P'} \subset R_{P}$ and if $(a',b') \in R_{P'}$ then $\{a',b'\} \cap \{a,b\} =\emptyset$ and either $a', b' \prec a$ or $a', b' \npreceq b$. So, by inspection, the eigenvalues of $M^{P'}$ satisfy the conditions $(a)$ and $(b)$ from the definition of the upset property.

\end{proof}



\section{Partition function and convergence rates} \label{S:convergencesec}


The stationary distribution for the promotion Markov chain is given by Theorem~\ref{stationary}. Here we find the partition function in the case when $P$ is a union of ordinal sums of forests and ladders. 

\begin{thm}
Let $P = F_1 \oplus L_1 + \cdots + F_k \oplus L_k$ be a poset of size $n$ where $F_i$ is a forest and $L_i$ is a ladder for $i = 1, \ldots, k$ of size $n$. Let $L_i = Q^{i}_{1} \oplus \cdots \oplus Q^{i}_{t_i}$ where $Q^{i}_{j} = \xymatrix{\bullet a^{i}_{j} & \bullet b^{i}_{j}}$ or $Q^{i}_{j} = \xymatrix{\bullet a^{i}_{j}}$.  The partition function for the promotion graph is given by 

\begin{equation}\label{Zp} Z_P = \displaystyle \prod_{i=1}^n \displaystyle x_{\preccurlyeq i} \displaystyle \prod_{Q^{i}_{j}: |Q^{i}_{j}| =2} \displaystyle \frac{x_{\preccurlyeq a^{i}_{j} \cup b^{i}_{j}}}{x_{\preccurlyeq a^{i}_{j}} + x_{\preccurlyeq b^{i}_{j}}},\end{equation}
where $x_{\preccurlyeq a^{i}_{j} \cup b^{i}_{j}} = \displaystyle \sum_{s \preccurlyeq a^{i}_{j} \text{ or } s \preccurlyeq b^{i}_{j}} x_s$. 
\end{thm}

\begin{proof}
By Theorem~\ref{stationary}, we need to show that $w'(\pi) := w(\pi) Z_P$ with \begin{equation}\label{wpi} w(\pi)  = \displaystyle \prod_{i=1}^n \displaystyle \frac{1}{x_{\pi_1} + \cdots + x_{\pi_i}}\end{equation} satisfies $\displaystyle \sum_{\pi \in \mathcal{L}(P)} w'(\pi) = 1.$ 

We will use induction on the size of $P$. One can readily check that this is true if $n=1$. Assume it is true for posets of this form of size $n-1$ and let $P$ be as described in the assumptions. If $\pi=\pi_{1}\cdots \pi_{n}$, then $\pi_{n}$ is an element in one of the top levels of $P$, i.e.,  $\pi_{n} \in Q^{i}_{t_i}$ for some $i \in [k]$. Therefore,

\[\displaystyle \sum_{\pi \in \mathcal{L}(P)} w'(\pi) = \displaystyle \sum_{i: |Q^{i}_{t_i}|=2} \left(\displaystyle \sum_{\sigma \in \mathcal{L}(P \setminus \{a^{i}_{t_{i}}\})} w'(\sigma a^{i}_{t_{i}})  + \displaystyle \sum_{\sigma \in \mathcal{L}(P \setminus \{b^{i}_{t_{i}}\} )} w'(\sigma b^{i}_{t_{i}}) \right) + \sum_{i: |Q^{i}_{t_i}|=1} \displaystyle \sum_{\sigma \in \mathcal{L}(P \setminus \{a^{i}_{t_{i}}\})} w'(\sigma a^{i}_{t_{i}})  .\]
By~\eqref{Zp} and~\eqref{wpi}, if $|Q^{i}_{t_i}| =2$ then 

\begin{align*}w'(\sigma a^{i}_{t_{i}}) = w'(\sigma) \displaystyle \frac{ x_{\preccurlyeq a^{i}_{t_{i}}}}{x_1 + \cdots + x_n} \cdot \frac{x_{\preccurlyeq a^{i}_{t_{i}} \cup b^{i}_{t_{i}}}}{x_{\preccurlyeq a^{i}_{t_{i}}} + x_{\preccurlyeq b^{i}_{t_{i}}}}, \\
w'(\sigma b^{i}_{t_{i}}) = w'(\sigma) \displaystyle \frac{ x_{\preccurlyeq b^{i}_{t_{i}}}}{x_1 + \cdots + x_n} \cdot \frac{x_{\preccurlyeq a^{i}_{t_{i}} \cup b^{i}_{t_{i}}}}{x_{\preccurlyeq a^{i}_{t_{i}}} + x_{\preccurlyeq b^{i}_{t_{i}}}},
\end{align*}
and if $|Q^{i}_{t_i}| =1$,
\[ w'(\sigma a^{i}_{t_{i}}) = w'(\sigma) \displaystyle \frac{ x_{\preccurlyeq a^{i}_{t_{i}}}}{x_1 + \cdots + x_n}.\]
Hence, using the induction hypothesis, we get
\begin{align*}\displaystyle \sum_{\pi \in \mathcal{L}(P)} w'(\pi)  & = \displaystyle \sum_{i: |Q^{i}_{t_i}|=2} \left(\displaystyle   \frac{ x_{\preccurlyeq a^{i}_{t_{i}}}}{x_1 + \cdots + x_n} \cdot \frac{x_{\preccurlyeq a^{i}_{t_{i}} \cup b^{i}_{t_{i}}}}{x_{\preccurlyeq a^{i}_{t_{i}}} + x_{\preccurlyeq b^{i}_{t_{i}}}}   + \displaystyle  \frac{ x_{\preccurlyeq b^{i}_{t_{i}}}}{x_1 + \cdots + x_n} \cdot \frac{x_{\preccurlyeq a^{i}_{t_{i}} \cup b^{i}_{t_{i}}}}{x_{\preccurlyeq a^{i}_{t_{i}}} + x_{\preccurlyeq b^{i}_{t_{i}}}} \right) \\
& + \displaystyle \sum_{i: |Q^{i}_{t_i}|=1} \displaystyle  \displaystyle \frac{ x_{\preccurlyeq a^{i}_{t_{i}}}}{x_1 + \cdots + x_n}  \\
& = \displaystyle \sum_{i: |Q^{i}_{t_i}|=2}  \displaystyle  \frac{x_{\preccurlyeq a^{i}_{t_{i}} \cup b^{i}_{t_{i}}}}{x_1 + \cdots + x_n}  + \displaystyle \sum_{i: |Q^{i}_{t_i}|=1} \displaystyle \frac{ x_{\preccurlyeq a^{i}_{t_{i}}}}{x_1 + \cdots + x_n} \\
& = 1.
\end{align*}

\end{proof}

\begin{exa}
Let $P$ be the poset from Example~\ref{promotionex}. $P$ can be written as $P = F \oplus L$ for $F = {\scalebox{.5}{ \xymatrix{ 2\ar@{-}[d] & 3 \\ 1 }}}$ and $L = {\scalebox{0.5}{ \xymatrix{5\ar@{-}[dr] & & 6 \ar@{-}[dl] \\ & 4 & }}}$.

The first product in the formula~\eqref{Zp} for $Z_{P}$ is 

\[ \displaystyle \prod_{i=1}^n \displaystyle x_{\preccurlyeq i} = x_1(x_1+x_2)x_{3}(x_1+x_2+x_3+x_4)(x_1+x_2+x_3+x_4+x_5)(x_1+x_2+x_3+x_4+x_6)\]

while the second product is 
\[\prod_{Q^{i}_{j}: |Q^{i}_{j}| =2} \displaystyle \frac{x_{\preccurlyeq a^{i}_{j} \cup b^{i}_{j}}}{x_{\preccurlyeq a^{i}_{j}} + x_{\preccurlyeq b^{i}_{j}}} = \frac{x_{1}+x_{2}+x_{3}+x_{4}+x_{5}+x_6}{(x_{1}+x_{2}+x_{3}+x_{4}+x_{5}) + (x_{1}+x_{2}+x_{3}+x_{4}+x_{6})}.\]


%

\end{exa}

For the case $P=F \oplus L$, we can make an explicit statement about the rate of convergence to stationarity and the mixing time. Let $P^k$ be the distribution after $k$ steps and  $\mathbb{P}^{k}$ be the $k$-th convolution power of the distribution $\mathbb{P}$. The rate of convergence is the total variation distance from stationarity after $k$ steps, that is,
\[ \|P^k - w\|_{TV}= \frac{1}{2} \sum_{\pi \in \mathcal{L}(P)} |\mathbb{P}^k(\pi) - w(\pi)|  \]
where $w$ is the stationary distribution. We will use the following theorem.

\begin{thm}\cite{sandpile} \label{leftwalk} Let $M$ be a monoid acting on a set $\Omega$ and let $\mathbb{P}$ be a probability distribution on $M$. Let $\mathcal{M}$ be the Markov chain with state set $\Omega$ such that the transition probability from $x$ to $y$ is the probability that $mx = y$ if $m$ is chosen from $M$ according to $\mathbb{P}$. Assume that $\mathcal{M}$ is irreducible and aperiodic with stationary distribution $w$ and that some element of $M$ acts as a constant map on $\Omega$.
Letting $P^{k}$ be the distribution of $\mathcal{M}$ after $k$ steps and $\mathbb{P}^{k}$ be the $k$-th convolution power of $\mathbb{P}$, we have that
\[\|P^k - w\|_{TV} \leq \mathbb{P}^{k}(M \setminus C),\]
where C is the set of elements of M acting as constants on $\Omega$.
\end{thm}

In our case the monoid (set with an associative multiplication and an identity element) acting on $\mathcal{L}(P)$ is $\mathcal{M}^{\hat{\partial}}$ generated by the operators $\hat{\partial}_{i}$ defined by the promotion  graph $G_{P}$. That is, for $\pi, \pi' \in \mathcal{L}(P)$, $\hat{\partial}_{i} \pi = \pi'$ if and only if $\pi' = \partial_{\pi^{-1}(i)} \pi$. In what follows it will be helpful to have the following alternate description of $\hat{\partial}_{i}$.

\begin{lem}\label{linearextensionsnew}
Let $P=F_{1}\oplus L_{1} + \cdots + F_{k}\oplus L_{k}$, where $F_{i}$ is a forest and $L_{i}$ is a ladder.  For $\pi \in \mathcal{L}(P)$,  $\hat{\partial}_k\pi$ is the linear extension of $P$ obtained from $\pi$ by moving the letter $k$ to the last position and reordering the letters $j \succeq k$, swapping the original order of incomparable elements at the same level of a ladder $L_{i}$.
\end{lem}

\begin{proof} 
By the definition of $\hat{\partial}_i$, we have $\hat{\partial}_i\pi = \tau_{n-1} \cdots \tau_{k+1} \tau_k \pi$, where $k=\pi^{-1}(i)$. The transpositions start swapping $i$ with the elements that follow it until an element $j\succeq i$ is reached. Then $j$ is swapped with the elements that follow it, etc. So, the elements $j$ that begin the new series of swaps are the ones that are in the ladder above $i$. Moreover, the two elements in this ladder will be swapped themselves because they are incomparable. 
\end{proof}

\begin{exa} Let $P$ be the poset on $[9]$ with covering relations $1 \prec 2$, $2 \prec 4$, $3 \prec 4$, $4 \prec 5$, $4 \prec 6$, $7 \prec 8$, and $7 \prec 9$. To compute $\hat{\partial}_3 371824695$, we first move $3$ to the end of the word to obtain $718246953$. Then we reorder the elements $\{3,4,5,6\}$ to form a linear extension, but in the process we swap the order of $5$ and $6$. So, since $6$ appears to the left of $5$ in $371824695$, we now place $5$ to the left of $6$. This way we get $\hat{\partial}_3 371824695 = 718234956$.
\end{exa}

For $x \in \mathcal{M}^{\hat{\partial}}$, let $\mathrm{im}(x) = \{x\pi \colon \pi \in \mathcal{L}(P)\}$. Let $\mathrm{rfactor}(x)$ be the maximal common right factor of the elements in $\mathrm{im}(x)$ and let $\mathrm{Rfactor}(x) = \{i \colon i \in \mathrm{rfactor}(x)\}$ . 

\begin{lem}\label{Rfactor}
Let $P = F \oplus L$ be a poset of size $n$, where $F$ is a  rooted forest and $L$ is a ladder. Then
\begin{enumerate}[(a)]
\item  $\mathrm{Rfactor}(x) \subseteq  \mathrm{Rfactor}(\hat{\partial}_i x)$ for all $x \in \mathcal{M}^{\hat{\partial}}$ and $i = 1, \ldots, n$,
\item $\mathrm{Rfactor}(x) \subsetneq \mathrm{Rfactor}(\hat{\partial}_kx)$ for $k$ maximal in $P\setminus \mathrm{Rfactor}(x)$. 
\end{enumerate}
\end{lem}

\begin{proof}
 Let $x \in \mathcal{M}^{\hat{\partial}}$. Each $\pi \in \mathrm{im}(x)$ is of the form $\pi = \pi' \mathrm{rfactor}(x)$. We consider two cases. If $i \in \mathrm{Rfactor}(x)$, then $\hat{\partial}_i \pi = \pi' \hat{\partial}_i \mathrm{rfactor}(x)$ and therefore, clearly, $\mathrm{Rfactor}(x) \subseteq  \mathrm{Rfactor}(\hat{\partial}_i x)$. Suppose now $i \notin \mathrm {Rfactor}(x)$. Since $\mathrm{Rfactor}(x)$ is an upset of $P$, the poset $P \setminus \mathrm{Rfactor}(x)$ is also of the form $P \setminus \mathrm{Rfactor}(x) = F' \oplus L'$, for a forest $F'$ and a ladder $L'$. Notice that if $P \setminus \mathrm{Rfactor}(x)$ has one maximal element then we get a contradiction of the maximality of $\mathrm{rfactor}(x)$. Therefore,  either $L' = \emptyset$ or $L' \neq \emptyset$ and $P \setminus \mathrm{Rfactor}(x)$ has two maximal elements. If $L' = \emptyset$, i.e., $P \setminus \mathrm{Rfactor}(x)$ is a forest, for every $i \in P \setminus \mathrm{Rfactor}(x)$, the set $\{j \in P \setminus \mathrm{Rfactor}(x) \colon i \preceq j\}$ is a chain and has a unique maximal element $k_{i}$. Then, by Lemma~\ref{linearextensionsnew}, $\mathrm{Rfactor}(x) \cup \{k_{i}\} \subset \mathrm{Rfactor}(\hat{\partial}_i x)$. On the other hand, if $L' \neq \emptyset$ and  $P \setminus \mathrm{Rfactor}(x)$ has two maximal elements, $a$ and $b$, then each $\pi \in \mathrm{im}(x)$ is of the form  $\pi = \pi''\;a\;b\; \,\mathrm{rfactor}(x)$ or $\pi = \pi''\;b\;a\; \mathrm{rfactor}(x)$ and both these forms appear in $\mathrm{im}(x)$. Hence $\mathrm{Rfactor}(\hat{\partial}_ax) \supseteq \mathrm{Rfactor}(x) \cup \{a\}$, $\mathrm{Rfactor}(\hat{\partial}_bx) \supseteq \mathrm{Rfactor}(x) \cup \{b\}$, and for $i \neq a,b$, $\mathrm{Rfactor}(\hat{\partial}_ix) = \mathrm{Rfactor}(x)$.
\end{proof}

\begin{thm}\label{convergence}
Let $P=F\oplus L$ be a poset of size $n$, where $F$ is a forest and $L$ is  a ladder. Let $p_{x} = \min\{x_i : 1 \leq i \leq n\}$. Then for $k \geq (n -1)/p_x$, the distance to stationarity of the promotion Markov chain satisfies $$\|P^k - \omega\|_{TV} \leq \mathrm{exp} \left(- \displaystyle \frac{(kp_x - (n -1))^2}{2kp_x} \right).$$
\end{thm}

\begin{proof}

For $m \in \mathcal{M}^{\hat{\partial}}$, let $u(m)= n - |\mathrm{Rfactor}(m)|$. The statistic $u$ has the following three properties:

\begin{enumerate}
\item[(1)] $u(m'm) \leq u(m)$ for all $m, m' \in \mathcal{M}^{\hat{\partial}}$;
\item[(2)] if $u(m) > 0$, then there exists $\hat{\partial}_{i} \in \mathcal{M}^{\hat{\partial}}$ such that $u(\hat{\partial}_{i}m) < u(m);$
\item[(3)]  $u(m) = 0$ if and only if $m$ acts as a constant on $\mathcal{L}(P)$. 
\end{enumerate}

The first two properties follow from Lemma~\ref{Rfactor}, while $u(m) = 0$ if and only if  $\mathrm{rfactor}(m)$ is a linear extension of $P$ which is equivalent to $m$ being a constant map. Furthermore, for the identity map $\epsilon$, $u(\epsilon) \leq n$.

A step $m_{i} \rightarrow m_{i+1}$ in the left random walk on $\mathcal{M}^{\hat{\partial}}$ is successful if  $u(m_{i+1}) < u(m_i)$. Property~$(1)$ of $u$ implies that the step is not successful if and only if
$u(m_i) = u(m_{i+1})$, and by Property~$(2)$, each step has probability at least $p_{x}$ to be successful. Therefore, the probability that $n \geq  u(m) > 0$ after $k$ steps of the left random walk on $\mathcal{M}^{\hat{\partial}}$ is bounded above by the probability of having at most $n-1$ successes in $k$ Bernoulli trials with success probability $p_{x}$. Using Theorem~\ref{leftwalk} and Chernoff's inequality,

\[\|P^k - \omega\|_{TV} \leq  \mathrm{exp} \left(- \displaystyle \frac{(kp_{x} - (n-1))^2}{2kp_{x}}\right),\]
where the inequality holds for $p_{x}k > n -1$.
\end{proof}

The \emph{mixing time}  is the number of steps $k$ until $\|P^k - \omega\|_{TV} \leq e^{-c}$.  Using Theorem~\ref{convergence}, it suffices to have 
\[ (kp_x - (n -1))^2 \geq 2kp_xc,\]
so the mixing time is at most $\frac{2(n+c-1)}{p_{x}}$. If the probability distribution $\{x_{i} \colon 1 \leq i \leq n\}$ is uniform, then $p_{x}$ is of order $\frac{1}{n}$ and the mixing time is of order at most $n^{2}$.

\section{Other posets} \label{S:conclusion}

The posets of the form $P = F_{1}\oplus L_{1}+ \cdots +F_{k}\oplus L_{k}$  discussed in this paper are not the only posets to have the nice property that the eigenvalues of their promotion matrices are linear in the $x_{i}$'s. In fact, we conjecture the following.   Let $A_{i}$ denote an antichain of size $i$.  

\begin{conj}
The characteristic polynomial for the promotion matrix $M^{P}$ of any  poset $P$ whose Hasse diagram is contained  in $A_k \oplus A_2$  factors into linear terms. 
\end{conj}

As a justification of this conjecture, we will prove the special case when $P$ is $A_k \oplus A_2$ with one missing edge. 

\begin{thm}\label{minusoneedge}
The characteristic polynomial of $M^{P}$ for the poset $P = (A_k \oplus A_2)\setminus \{(k,k+1)\}$ which is assumed to be labeled naturally is

\begin{equation} \label{edge} \det(M^P - \lambda I) = (x_{k+2} - \lambda)^{(k-1)!} \displaystyle \prod_{U \subseteq [k]} (x_U +x_{k+1} +x_{k+2} - \lambda)^{d_{k-|U|}} \prod_{U \subseteq [k-1]} (-x_U - \lambda)^{d_{k-|U|} + d_{k-|U|-1}},\end{equation}
where $x_{U}= \sum_{i \in U}x_{i}$ and $d_{i}$ is the number of derangements in the symmetric group $S_{i}$.

\end{thm}

\begin{proof}
Let $\sigma_i$ represent the permutation of $[k]$ in which $k$ is in the $i$-th position and $\sigma$ is the permutation of $[k-1]$ obtained when $k$ is deleted. Consider $M^{P}- \lambda I$ for $P = (A_k \oplus A_2)\setminus \{(k,k+1)\}$. 

$M^{P}$ can be split into blocks $B_{\sigma_i}^{\pi_j}$ of size $2 \times 2, 2 \times 3, 3 \times 2$, or $3 \times 3$ as follows: \\
\begin{itemize}
\item If $i = k$, the three rows of $B_{\sigma_k}^{\pi_j}$ correspond to the linear extensions $\sigma_k (k+1) (k+2)$, $\sigma_k (k+2) (k+1)$, and $\sigma (k+1) k (k+2)$. \\
\item If $i \neq k$, the two rows of $B_{\sigma_i}^{\pi_j}$ correspond to the linear extensions $\sigma_i (k+1) (k+2)$ and $\sigma_i (k+2) (k+1)$.
\end{itemize}
The columns of $B_{\sigma_i}^{\pi_j}$ are indexed analogously depending of whether $j=k$ or $j \neq k$. Let the transition matrix of the Tsetlin library for $k$ books be $M^{A_k} = \left(a_{\sigma_i}^{\pi_j}\right)_{i,j =1}^{k}$. Then  

$$B_{\sigma_i}^{\pi_j} = \begin{cases}  a_{\sigma}^{\pi}\begin{pmatrix} 0 & 0 &0 \\ 0 & 0 & 0 \\ 0 & 1 & 0 \end{pmatrix} + \delta_{\sigma,\pi} \begin{pmatrix} x_{k+2}-\lambda & x_{k+1} & x_k \\ x_{k} + x_{k+2} & x_{k+1}-\lambda & 0 \\ 0 & x_{k+1} & x_k + x_{k+2}-\lambda \end{pmatrix} & \text{ if } i, j = k,\\ 
a_{\sigma_i}^{\pi_j} \begin{pmatrix} 0 & 1 \\ 1 & 0 \end{pmatrix} + \delta_{\sigma_i, \pi_j} \ \begin{pmatrix} x_{k+2}-\lambda & x_{k+1} \\ x_{k+2} & x_{k+1}-\lambda \end{pmatrix} & \text{ if } i, j \neq k, \\ 
a_{\sigma_k}^{\pi_j}\begin{pmatrix} 0 & 1 \\ 1 & 0 \\ 0 & 0 \end{pmatrix}  & \text{ if } i=k, j \neq k, \\
a_{\sigma_i}^{\pi_k} \delta_{\sigma, \pi} \begin{pmatrix} 0 & 0 & 1 \\ 1 & 0 & 0 \end{pmatrix} & \text{ if } i \neq k, j =k, \end{cases}$$ where $\delta_{{x,y}}$ is the Kronecker delta function.


If we subtract the first two rows of $B_{\sigma_k}^*$ from $B_{\sigma_i}^*$ for all $i  \neq k$, the block change is given by 
\[B_{\sigma_i}^{\pi_j} \mapsto B_{\sigma_i}^{\pi_j} = \begin{cases} -  \delta_{\sigma,\pi} \begin{pmatrix} x_{k+2} -\lambda & x_{k+1} & 0 \\  x_{k+2} & x_{k+1} -\lambda& 0  \end{pmatrix} & \text{ if }  j =k, \\ a_{\sigma_i}^{\pi_j} \begin{pmatrix} 0 & 1 \\ 1 & 0 \end{pmatrix} + \delta_{\sigma_i
, \pi_j}  \begin{pmatrix} x_{k+2}-\lambda & x_{k+1} \\ x_{k+2} & x_{k+1}-\lambda \end{pmatrix}- a_{\sigma_k}^{\pi_j}\begin{pmatrix} 0 & 1 \\ 1 & 0  \end{pmatrix}  & \text{ if } j \neq k. \end{cases}\]

If we then add the columns of $B_*^{\pi_j}$ to the first two columns of $B_*^{\pi_k}$ for $j = 1, \ldots, k-1$, the block change is given by 
\[B_{\sigma_i}^{\pi_k} \mapsto B_{\sigma_i}^{\pi_k} = \begin{cases} \begin{pmatrix} 0 & 0 & 0 \\ 0 & 0 & 0 \end{pmatrix} & \text{ if } i \neq k,\\ a_{\sigma}^{\pi}\begin{pmatrix} 0 & 1 & 0 \\ 1 & 0 & 0 \\ 0 & 1 & 0 \end{pmatrix} + \delta_{\sigma_i,\pi_j} \begin{pmatrix} x_{k+2}-\lambda & x_{k+1} & x_k \\ x_{k} + x_{k+2} & x_{k+1}-\lambda & 0 \\ 0 & x_{k+1} & x_k + x_{k+2}-\lambda \end{pmatrix}  & \text{ if } i = k. \end{cases}\]

If the blocks are ordered so that $B_{\sigma_k}^{\pi_k}$ are in the upper left, this yields a block upper triangular matrix with one block $B_{u}$ of size $3(k-1)! \times 3(k-1)!$ consisting of the $3 \times 3$ blocks $B_{\sigma_k}^{\pi_k}$ and  another block $B_{\ell}$  consisting of the $2 \times 2$ blocks $B_{\sigma_i}^{\pi_j}$ for $i, j \neq k$. Next, we compute the determinant of each of these blocks separately. The upper block $B_u$ is similar in structure to $M^{A_{k-1}} - \lambda I$. Namely, $B_u$ can be obtained from $M^{A_{k-1}} - \lambda I$ by the substitutions \[x_m \mapsto \begin{pmatrix} 0 & x_m & 0 \\ x_m & 0 & 0 \\ 0 & x_m & 0 \end{pmatrix} \qquad \text{ and } \qquad - \lambda \mapsto \Lambda = \begin{pmatrix} x_{k+2} - \lambda & x_{k+1} & x_k \\ x_{k} + x_{k+2} & x_{k+1} - \lambda & 0 \\ 0 & x_{k+1} & x_{k} + x_{k+2} - \lambda \end{pmatrix}.\] In other words,
\[ B_{u} = M^{A_{k-1}} \otimes \begin{pmatrix} 0 & 1 & 0 \\ 1 & 0 & 0 \\ 0 & 1 & 0 \end{pmatrix} + I_{k-1} \otimes \begin{pmatrix} x_{k+2} - \lambda & x_{k+1} & x_k \\ x_{k} + x_{k+2} & x_{k+1} - \lambda & 0 \\ 0 & x_{k+1} & x_{k} + x_{k+2} - \lambda \end{pmatrix}.\]

By  \cite{brown2000semigroups}, $SM^{A_{k-1}}S^{-1}$ is diagonal for some matrix $S$. Therefore,  $(S \otimes I_3) B_u (S^{-1} \otimes I_3)$ is block diagonal with blocks $\begin{pmatrix} x_{k+2} - \lambda & x_{k+1} + x_{U} & x_k \\ x_k + x_{k+2} + x_{U} & x_{k+1} - \lambda & 0 \\ 0 & x_{k+1} + x_{U} & x_{k} + x_{k+2} - \lambda \end{pmatrix}$ corresponding to $(x_{U} - \lambda)$ in $M^{A_{k-1}} - \lambda I$. Thus,  \[ \det B_{u} = \displaystyle \prod_{U \subseteq [k-1]} (-x_{U} - \lambda)^{d_{k-1-|U|}}(x_{k} + x_{k+1} + x_{k+2} + x_{U} - \lambda)^{d_{k-1 - |U|}}(x_{k+2}  -\lambda)^{d_{k -1-|U|}}.\]

For the lower block, $B_{\ell},$  first notice that there are similarities between $M^P$ and $M^{A_k}$. First, the entries in $M^{A_k}= \left(a_{\sigma_i}^{\pi_j}\right)_{i,j =1}^{k}$ are only zero or $x_i$ for some $i \in \{1, \ldots, k\}$. We perform the following row and column operations on $M^{A_k} -\lambda I = \left(m_{\sigma_i}^{\pi_j}\right)_{i,j =1}^{k}$. If we subtract the rows ${\sigma_i}$ from ${\sigma_k}$ for all $i \neq k$ the entries change to 

\[ m_{\sigma_i}^{\pi_j} \mapsto m_{\sigma_i}^{\pi_j} = \begin{cases} -\lambda \delta_{\sigma, \pi} &  \text{ if } j = k, \\ a_{\sigma_i}^{\pi_j} - a_{\sigma_k}^{\pi_j} +\lambda\delta_{\sigma_i, \pi_j} & \text{ if }  j \neq k. \end{cases} \] If we then add the columns ${\pi_j}$ to the column ${\pi_k}$ for $j =1,\ldots, k$, the entries become 

\[m_{\sigma_i}^{\pi_k} \mapsto m_{\sigma_i}^{\pi_k} = \begin{cases} 0 & \text{ if  } i \neq k, \\ a_{\sigma_i}^{\pi_j} + \lambda\delta_{\sigma_i, \pi_j} & \text{ if } i =k.\end{cases}\]

Notice that the matrices for the row and column operations are inverses of each other, so that the resulting matrix is similar to $M^{A_k} -\lambda I$. Moreover,  if we order the linear extensions of $A_k$ so that $a_{\sigma_k}^{\pi_k}$ is in the upper left corner, the resulting matrix is  block upper triangular matrix, where for the lower block, $b_{\ell}$, we have \[B_{\ell} = b_{\ell} \otimes \begin{pmatrix} 0 & 1 \\ 1 & 0 \end{pmatrix} + I \otimes \begin{pmatrix} x_{k+2} & x_{k+1} \\ x_{k+2} & x_{k+1} \end{pmatrix}.\] The block $b_{\ell}$ does not contain $x_{k}$ and is of size $(k-1)(k-1)! \times (k-1)(k-1)!$, which is the sum of the multiplicities of all the eigenvalues of $M^{A_{k}}$ whose support does not contain $x_{k}$. 

Since $M^{A_k} -\lambda I$ is diagonalizable, $b_{\ell}$ is also diagonalizable. Let  $S$ be such that $Sb_{\ell}S^{-1}$ is diagonal.  Then $(S \otimes I_{2})B_{\ell} (S^{-1} \otimes I_{2})$ is a block diagonal matrix with blocks $\begin{pmatrix} x_{k+2} -\lambda & x_{k+1} + x_{U} \\ x_{k+2} + x_{U} & x_{k+1} - \lambda \end{pmatrix}$ for every  eigenvalue $x_{U}$ of $M^{A_k}$ whose support does not contain $x_{k}$. This gives 

\[\det B_{\ell} = \displaystyle \prod_{U \subseteq [k-1]} (-x_{U}- \lambda)^{d_{k - |U|}}(x_{k+1} + x_{k+2} + x_{U} - \lambda)^{d_{k - |U|}}.\]

Since $\det(M^P - \lambda I) = \det B_{u} \det B_{\ell}$, we get~\eqref{edge}.
\end{proof}

\begin{exa} 

Let $P = (A_2 \oplus A_2)\setminus \{(2,3)\}$, i.e., $P$ is the poset on $[4]$ with covering relations $1 \prec 3$, $1 \prec 4$, and $2 \prec 4$. The eigenvalues of $M^P$ are $$x_1 + x_2 + x_3+x_4, 0, x_3+x_4, -x_1, x_4.$$ 
\end{exa}


\bibliographystyle{alpha}



\end{document}